\documentclass{amsart}
\usepackage{chngcntr}
\usepackage{apptools}
\usepackage{color}
\usepackage{amsthm,amssymb,verbatim}
\usepackage{graphicx}
\usepackage{enumerate}
\usepackage{enumitem}
\usepackage{chngcntr}
\usepackage{apptools}
\usepackage{color}
\usepackage{amsthm,amssymb,verbatim}
\usepackage{graphicx}
\usepackage{enumerate}
\newcommand{\Sing}{\operatorname{Sing}}

 \newcommand{\Cc}{\mathcal C}
 
 \newcommand{\Nn}{\mathcal{N}}

  \newcommand{\MM}{\mathcal{M}}
  
\newcommand{\Pp}{\mathcal{P}}

 \newcommand{\RR}{\mathbf{R}}  
 \newcommand{\BB}{\mathbf{B}}  
    
  \newcommand{\Div}{\operatorname{Div}}
  
    \newcommand{\dist}{\operatorname{dist}}

 \newcommand{\eps}{\epsilon}

 \newcommand{\UU}{\mathcal{U}}
 \newcommand{\VV}{\mathcal{V}}

\newcommand{\vv}{\mathbf v}

\newcommand{\spt}{\operatorname{spt}}
\newcommand{\Hh}{\mathcal{H}}
\usepackage{amsthm}

\newcommand{\interior}{\operatorname{interior}}
\input epsf
\def\begfig {
\begin{figure}
\small }
\def\endfig {
\normalsize
\end{figure}
}

    \newtheorem{theorem}    {Theorem}   
    \newtheorem{lemma}      [theorem]       {Lemma}
    
    \newtheorem{proposition}       [theorem]       {Proposition}
    
    \newtheorem{claim}{Claim}
        \newtheorem*{claim*}{Claim}
    \newtheorem*{theorem*}{Theorem}
    \theoremstyle{definition}
    \newtheorem{definition}  [theorem] {Definition}
     
    \theoremstyle{definition}
    \newtheorem{remark}   [theorem]       {Remark}

\title[Non-fattening of mean curvature flow]{Non-fattening of mean curvature flow at singularities of mean convex type}
\author{Or Hershkovits}
\thanks{The first author was partially supported by an AMS-Simons Travel Grant}
\address{Department of Mathematics\\ Stanford University\\ Stanford, CA 94305}
\email{orher@stanford.edu}
\author{Brian White}
\thanks{The second author was partially supported by grants from the Simons Foundation
(\#396369) and from the National Science Foundation (DMS~1404282, DMS~ 1711293).}
\address{Department of Mathematics\\ Stanford University\\ Stanford, CA 94305}
\email{bcwhite@stanford.edu}
\date{February 3, 2018}
\usepackage{hyperref}
\usepackage[alphabetic, msc-links, backrefs]{amsrefs}
\begin{document}
\maketitle
\begin{abstract}
We show that a mean curvature flow starting from a 
compact, smoothly embedded hypersurface $M\subseteq  \RR^{n+1}$ remains unique past singularities, provided the
 singularities are
 of mean convex type, i.e., if around each singular point,
  the surface moves in one direction. 
Specifically, the level set flow of $M$ does not fatten if all singularities are of mean convex type.  We further show that assumptions of the theorem hold provided all blow-up flows are of the kind appearing in a mean convex flow, i.e. smooth, multiplicity one and convex.
Our results generalize the well known fact that the level set flow of a mean convex initial hypersurface $M$ does not fatten.  They also provide the first instance where non-fattening is concluded from local information around the singular set or from information about the singularity profiles of a flow. \end{abstract}
\section{Introduction}
It is an old idea in geometric analysis, and PDEs in general, to separate the questions of existence and regularity; one is often led to defining a weak notion of solution, the existence of which can be shown by one set of ideas, while studying its properties may require different methods. In the study of mean curvature flow, one very 
useful notion of  weak solution is that of the level set flow, 
introduced numerically in \cite{OsherSeth} 
and developed rigorously in \cites{evans-spruck,CGG}.
Given a closed set $X\subseteq  \RR^{n+1}$, its level set flow 
$t\in [0,\infty)\mapsto F_t(X)$ is a 
one-parameter family of closed sets starting at $F_0(X)=X$ and 
 satisfying the avoidance principle: $F_t(X)\cap M(t)=\emptyset$, provided 
 $t\in[a,b]\mapsto M(t)$ is a smooth mean curvature flow 
 with $[a,b]\subseteq[0,\infty)$ and with $M(a)\cap F_a(X)=\emptyset$. 
 Indeed, the level set flow is fully characterized as the maximal family of sets  satisfying the two properties above \cites{Ilmanen_LS,Ilmanen,white_topology}. 
Ideally, weak solutions should coincide with smooth solutions whenever the latter exist. 
In our case, if $t\in [0,T) \mapsto M(t)$ is a smooth mean curvature flow of 
closed, embedded hypersurfaces in $\RR^{n+1}$, then 
 $F_t(M_0)=M(t)$ for every $0\leq t<T$, as was shown in \cites{evans-spruck,CGG}. 
Although in many regards the level set flow resembles mean curvature flow of smooth surfaces, 
it was observed already in the original paper \cite{evans-spruck} that if $X$ is a smooth closed
planar curve that crosses itself, then $F_t(X)$ will instantly develop an interior.
In general, if the interior of $F_t(X)$ is empty for $t=0$ and nonempty at some later time, 
we say that $X$ fattens under the level set flow.
Even if the initial hypersurface is smooth and embedded, 
fattening can occur after the surface becomes singular, as described in \cite{White_ICM}. 
Although the level set flow is unique, 
the fattening phenomenon is related to non-uniqueness for other weak formulations of mean curvature
flow.  For example, let $M\subset\RR^{n+1}$ be a smooth, closed hypersurface.  Let $U$ be the compact
region it bounds.  Then 
\begin{equation}\label{eq:three}
\begin{aligned}
&t\in [0,\infty) \mapsto M_{\rm outer}(t) := \partial F_t(U), \\
&t\in [0,\infty) \mapsto M_{\rm inner}(t):= \partial F_t(\overline{U^c}),\, \text{and} \\
&t\in [0,\infty) \mapsto F_t(M)
\end{aligned}
\end{equation}
all may be 
regarded
as weak versions of mean curvature flow starting from $M$.
In particular, if the flow $M_{\rm inner}(\cdot)$ or $M_{\rm outer}(\cdot)$ is smooth in some region of spacetime, then it is indeed ordinary mean curvature flow in that region.  
If $F_t(M)$ has interior,
then it differs from $M_\textnormal{inner}(t)$ and $M_\textnormal{outer}(t)$, since
neither of those sets has interior.  One can also show, in this case, that 
   $M_\textnormal{inner}(t)\ne M_\textnormal{outer}(t)$.
 Thus if $M$ fattens, then $F_t(M)$, $M_\textnormal{inner}(t)$ 
 and $M_\textnormal{outer}(t)$ are three distinct flows.
 
 (The flow $t\mapsto \partial F_t(U)$ is  somewhat awkward to work with
because it need not trace out a closed subset of spacetime.  For this reason, it 
is actually better to define
  $M_\textnormal{outer}(t)$ to be $\lim_{\tau\uparrow t} \partial F_\tau(U)$.
The two definitions agree except at a countable set of times;
see Theorems~\ref{two-defs-theorem} and~\ref{countable-times-theorem}.
The same remarks apply to $M_\textnormal{inner}(\cdot)$.)
   
Extending the work of Brakke, Ilmanen introduced a notion of ``matching Brakke flow"~\cite{Ilmanen}.
He proved that if $M$ is non-fattening, then there is a unique matching Brakke flow
$t\mapsto M(t)$ with $M(0)=M$.  We conjecture that the converse is true.
(Indeed, we believe that the flows $M_{\rm inner}(\cdot)$ and $M_{\rm outer}(\cdot)$
both can be identified with matching Brakke flows.)  If this conjecture is true,
nonfattening of level set flow would be equivalent to uniqueness for matching Brakke flow.

In addition to being a fundamental question about the nature of the flow, whether or not fattening occurs is extremely relevant to the regularity theory for MCF and for potential geometric applications. From the regularity theory point of view, non-fattening ensures that the assumptions of Brakke regularity theorem \cite{Bra} hold at almost every time, and, in particular, that almost every time slice of a non-fattening MCF is regular $\mathcal{H}^n$ almost everywhere. From the point of view of geometric applications, non-fattening corresponds to continuous dependence on initial data, which is key if one wants to apply weak MCF on families of initial configurations in order to study the space of embedded hypersurfaces of a fixed topological type (compare \cite{Bam_Kle}). 

In light of the above, it is desirable to find conditions that prevent fattening. 
We have already mentioned that a smooth hypersurface
cannot fatten until after singularities form.
 Short-time non-fattening for initial sets satisfying a Reifenberg condition with small Reifenberg parameter was established by the first author in \cites{Her_Reif} (see also \cite{Her_Reif_High} for the higher co-dimension surfaces). 
In that case, the flow immediately becomes smooth (though it may later develop singularities),
and the non-fattening follows from
short-time existence of smooth flows (with suitable estimates) serving as barriers to the level set flow. 
In the presence of singularities, two initial conditions are known to imply non-fattening for all time: 
 the star-shapedness of $M$  \cite{soner_star}  and  mean convexity of $M$ \cite{evans-spruck}.
(See also \cite{white_size} for a more geometric proof that mean convex sets do not fatten.)       
The facts that surfaces can fatten only after they become singular and that mean convex surfaces
never fatten suggest the following conjecture:
\begin{quote}
An evolving surface cannot fatten unless it has a singularity with no spacetime neighborhood in
which the surface is mean convex.
\end{quote}
According to the conjecture, to ensure nonfattening, 
we do not need mean convexity everywhere; it suffices to have it near
the singularities.
In this paper, we prove a precise formulation of the conjecture. 
Additionally, we show that having a spacetime neighborhood in which the surface is mean convex can be concluded from data about all blowup limit flows: if all blowup limit flows at a singular points are smooth, multiplicity one, and convex, then there is a spacetime neighborhood of the point in which the flow is mean convex. In particular, this means that fattening can not occur if all singular points have the blowup behavior described above.

\section{The main result}
Before stating our theorem, we need some definitions.
\begin{definition}
Let $M$ be a compact, smoothly embedded hypersurface. 
The {\bf fattening time} of the level set flow of $M$ is
\begin{equation*}
T_\textnormal{fat}:=\inf\{t>0: \text{$F_t(M)$ has non-empty interior}\}.
\end{equation*}
\end{definition}
As the fattening time is a rather illusive quantity to work with directly, 
we will work with a different quantity which we call the \textit{discrepancy time} and which bounds the fattening time from below.
  
Let $U$ be the compact region bounded by a compact, smoothly embedded hypersurface $M$, and let $U'=\overline{U^c}$. Let $U(t)$ and $U'(t)$ denote the
corresponding level set flows:
\begin{align*}
U(t)&:=F_t(U),
\\
U'(t)&:=F_t(U'),
\end{align*}
and let $\UU$ and $\UU '$ be the space time tracks:
\begin{align*}
\UU &:=\{(x,t)\subseteq \RR^{n+1}\times \RR\;|\;x\in U(t)\},
\\
\UU '&:=\{(x,t)\subseteq \RR^{n+1}\times \RR\;|\;x\in U'(t)\}.
\end{align*}
We let 
\begin{equation}\label{eq:definition-inner-outer}
 \begin{aligned}
 M(t)&:=\{x\in \RR^{n+1}\;|\;(x,t)\in\partial \UU \},
 \\
 M'(t)&:=\{x\in \RR^{n+1}\;|\;(x,t)\in\partial \UU'\}.
 \end{aligned}
 \end{equation}
 (Here $\partial \UU$ and $\partial \UU'$ are the relative boundaries of $\UU$ and $\UU'$
 as subsets of $\RR^{n+1}\times [0,\infty)$.)
 We say that $t\mapsto M(t)$ and $t\mapsto M'(t)$ are the {\bf outer} and {\bf inner}
 flows of $M$.
 
The surfaces $M(t)$ and $\partial U(t)$ are closely related.
Trivially $\partial U(t)\subseteq M(t)$.  (This uses nothing except that $\UU$ is 
a closed subset of spacetime.)  Furthermore, 
\begin{equation*}\label{eq:two-defs}
 M(t) =  \lim_{\tau\uparrow t}\partial U(\tau)  \\
\end{equation*}
for all $t>0$, and $M(t) =  \partial U(t)$ for all but countably many $t$.
  See Theorems~\ref{two-defs-theorem} and~\ref{countable-times-theorem} in
  the appendix.
Then $M(t),M'(t)\subseteq F_t(M)$ for all $t\ge 0$
 (see appendix and Proposition~\ref{inner_most_prop} in particular). The discrepancy time is the first time at which those three flows start to differ.
\begin{definition}
The \textbf{discrepancy time} is
\begin{equation*}\label{eq:start-to-differ}
  T_{\rm disc} 
  = 
  \inf \{ t>0: \text{$M(t)$, $M'(t)$, and $F_t(M)$ are not all equal} \}.
\end{equation*}
\end{definition}
Thus $M(t)=M'(t)=F_t(M)$ for $t\le T_{\rm disc}$.
\begin{remark}\label{fat_vs_disc_0}
One always has $T_\textnormal{fat} \geq T_\textnormal{disc}$. 
Indeed, If $F_t(M)$ has an interior point $x$, 
then by inclusion of evolving balls, we see that $(x,\tau)\in \textrm{interior}(\UU)$ 
for every $\tau>t$  
sufficiently close to $t$. Thus $x\notin M(\tau)$ for such $\tau$, so in particular 
$\tau \geq T_\textnormal{disc}$. 
\end{remark}
We conjecture that $T_\textnormal{fat}=T_\textnormal{disc}$ for 
every smooth initial surface, but proving it would require a major advance in our
knowledge of mean curvature flow regularity. 

We next fix the notion of points of mean convex/mean concave type.
\begin{definition}\label{concave type_def}
Let $t>0$.
A point $x\in M(t)$ 
is called of {\bf mean convex type} (resp.\ {\bf mean concave type})
 if there exists an $\eps>0$ such that for every $t-\eps\leq t_1<t_2\leq t$,
\begin{equation*}
\begin{aligned}
U(t_2)\cap B(x,\eps) &\subseteq  U(t_1)\setminus M(t_1)
\\
(\text{resp.\ }\; 
U(t_1)\cap B(x,\eps) &\subseteq  U(t_2) \setminus M(t_2)).
\end{aligned}
\end{equation*} 
\end{definition}
\begin{remark}
Because the arguments in \cite{white_nature,white_size} are local, 
the regularity results therein hold for flows for which all the singularities occur at mean convex/mean concave points. In particular, the parabolic Hausdorff dimension of the spacetime singular set is at most $n-1$
(By more recent work of Colding and Minicozzi \cite{cold_min}, 
the spacetime singular set even has 
finite $(n-1)$ dimensional parabolic Hausdorff measure.)
If $n<7$, then the tangent flows are shrinking spheres or cylinders, and more general blowups
must be convex and smooth. 
However, none of those results are needed in this paper.
\end{remark}
\begin{definition}
We say that a spacetime point $(x,t)$ with $t>0$ is  {\bf regular} for the flow $t\in[0,\infty)\mapsto M(t)$ 
if there is an $\eps>0$ such that 
\[
\tau\in [t-\eps,t+\eps] \mapsto \BB(x,\eps)\cap M(t)
\]
is a smooth
mean curvature flow of smoothly embedded hypersurfaces; if there is no such $\eps$, we say that
$(x,t)$ is {\bf singular}.
We say that $(x,t)$ is {\bf backwardly regular} for the flow if there is an $\eps>0$ such that 
\[
\tau\in [t-\eps,t]\mapsto \BB(x,\eps)\cap M(\tau)
\]
is a smooth mean curvature flow of smoothly
embedded hypersurfaces in $\BB(x,\eps)$; if there is no such $\eps$, we say that $(x,t)$ is {\bf backwardly singular}.
\end{definition}
We will also sometimes write ``$x\in M(t)$ is regular (singular, backwardly regular, backwardly singular)" to
mean ``$x\in M(t)$ and 
   $(x,t)$ is a regular (singular, backwardly regular, backwardly singular) point".
We can now state our main theorem:
\begin{theorem}\label{main_theorem}
Let $M\subseteq\RR^{n+1}$ be a compact, smoothly embedded hypersurface. 
If $T_{\textnormal{disc}}<\infty$,
 then there exists a backwardly singular point $x\in M(T_{\textnormal{disc}})$  
 that is neither of mean convex nor of mean concave type.
Equivalently, suppose that $0<T\le T_{\textnormal{disc}}$, and suppose that all the backward singularities at time $T$
are of mean convex or mean concave type.  Then $T<T_\textnormal{disc}$.
\end{theorem} 
Note that no assumption is made on the behavior at times later than $T_\textnormal{disc}$
(in the first formulation) or than $T$ (in the second formulation). 
An immediate corollary of Theorem \ref{main_theorem} and Remark \ref{fat_vs_disc_0} is the following theorem, which confirms the conjecture appearing in the introduction.
\begin{theorem}\label{no_fat_thm}
Let $M\subseteq\RR^{n+1}$ be a compact, smoothly embedded hypersurface.  Let $T>0$ and suppose that for every $t\in [0,T]$, all the backward singularities of $M(t)$ are of mean convex or mean concave type. Then $T< T_{\textnormal{fat}}$.
\end{theorem}
 
\begin{remark}
Theorem~\ref{main_theorem}, its proof, and its 
corollary, Theorem ~\ref{no_fat_thm}, remain valid in any smooth ambient Riemannian manifold,
provided we assume that
$F_t(U)$ is compact
for $t\le T_\textnormal{fat}$ (in the first formulation of the theorem)
or for $t\le T$ (in the second formulation).
To ensure such compactness, it is enough to assume that the ambient manifold is complete
with Ricci curvature bounded below: that assumption implies that compact sets remain compact for all time
under the level set flow~\cite{Ilmanen_Ind}.
\end{remark}

The hypothesis that each singularity have mean convex (or mean concave) type may, at first glance, seem 
quite restrictive, in that it requires information about an entire (backward) space time neighborhood of the singularity. However, existence of such a neighborhood follows
from a more infinitesimal hypothesis, namely the hypothesis that all blow-ups (i.e, singularity models)
at the spacetime point are of the kind appearing in a mean convex (or mean concave) flows.

\begin{definition}\label{blow-up-type}
A singular point $(x_0,t_0)$ with $t_0\le T_{\rm disc}$
  is called of \textbf{mean convex blow-up type}  if 
\begin{enumerate}
\item\label{blow-up-type-1}
  All tangent flows at $(x_0,t_0)$ are smooth, multiplicity-one cylinders or spheres.
\item\label{blow-up-type-2}
 Whenever $(x_k,t_k)\rightarrow (x_0,t_0)$ are regular points with $t_k\leq t_0$, then 
 the norm $Q_k:=|A(x_k,t_k)|$ of the second fundamental form of $M(t_k)$ at $x_k$ 
tends to infinity, and, after passing to a subsequence, the flows
\[
   t \in [-Q_k^2\, t_k,0] \mapsto Q_k(U(t_k + Q_k^{-2}t)-x_0) \\
\]
 converge smoothly to a flow $t\in(-\infty,0]\mapsto N_t$ where the $N_t$
 are convex regions with smooth boundaries. 
\end{enumerate}
We say that $(x_0,t_0)$ is of \textbf{mean convave blow-up type} if~\eqref{blow-up-type-1}
and~\eqref{blow-up-type-2} 
hold with $U'$ in place of $U$.
\end{definition}

\begin{theorem}\label{alt}
Suppose that $(x_0,t_0)$ with $t_0\le T_{\rm disc}$ is a singularity of mean convex (resp. mean concave)
blowup type.  Then $(x_0,t_0)$ is of mean convex (resp. mean concave) type.
\end{theorem}

The proof of this theorem occupies Section \ref{blowupimplies}. A very interesting open problem is whether condition~\eqref{blow-up-type-2} in Definition~\ref{blow-up-type}
is superfluous.  In other words, does having a multiplicity-one cylindrical tangent flow at a singularity
imply that the singularity is of mean convex (or concave) blow-up type,
and therefore, by Theorem~\ref{alt},
of mean convex (or concave) type?

\begin{remark}
In the global setting, being mean convex on the regular part and having only singular points of mean convex blow-up type  is equivalent to mean convexity of the flow. Indeed, Theorem~\ref{alt} implies that those assumptions imply mean convexity, while the other (harder) direction follows from \cite{white_size,white_nature} in Euclidean space, and from~\cite{HH} in Riemannian manifolds.
\end{remark}

The following theorem follows immediately by combining Theorems~\ref{no_fat_thm} and~\ref{alt}.
\begin{theorem}\label{sing_prof_thm}
Let $M\subseteq\RR^{n+1}$ be a compact, smoothly embedded hypersurface, and let $T>0$. Assume that all singular points on $[0,T]$ are either of mean convex blow-up type, or of mean concave blow-up type. Then all of the singular points on $[0,T]$ are of mean convex or mean concave type. In particular $T< T_{\textnormal{disc}} \leq T_{\textnormal{fat}}$.
\end{theorem}

\begin{remark}
It is instructive to compare Theorem ~\ref{sing_prof_thm} to the recent fantastic result of Bamler and Kleiner showing uniqueness of weak Ricci flow in dimension three \cite{Bam_Kle} (The notion of weak Ricci flow in dimension three was introduced earlier in \cite{KL_weak}) . While our proof and theirs are very different, it is interesting to note that, in both instances, information about the structure of all blow ups, and not just self-shrinking solutions, was required. Furthermore, the blow up behavior of mean convex MCF has much similarity to the blow up behavior of $3$-dimensional Ricci flow. Consequently, the structure of blow-ups assumed is Theorem \ref{sing_prof_thm} is the extrinsic analog of $\kappa$-solutions of the Ricci flow, which is the asymptotic assumption built into the notion of weak Ricci flow.

One difference in the theories is that for $3$-dimensional Ricci flow, {\em all} singularities
have the required blow-up type, whereas
for mean curvature flow of $m$-dimensional initially smooth hypersurfaces, 
this is not the case (except when $m=1$).

\end{remark}

\newtheorem*{idea}{Idea of the proof of Theorem~\ref{main_theorem}}
\begin{idea}
We  sketch the proof of the second formulation of the 
main theorem (Theorem~\ref{main_theorem}).
 For simplicity, assume all the singularities are of mean convex (and not mean concave) type and that $T=0$ (so that the initial time of the flow is negative).
 In a neighborhood $Y$ of the singularities of $M(0)$,
 we construct a time of arrival function $u$ 
 for the evolution $t\mapsto M(t)$ in some small time interval $(-\delta,\delta)$. 
 (Near the singularities, time of arrival makes sense as a single valued function because the
 surfaces are moving in one direction.)
 The zero set of $u(\cdot)-t$ is $M(t)\cap Y$,
 and the set where $u(\cdot)- t \ge 0$ is $U(t)\cap Y$.  
 By interpolating between the function $(x,t)\mapsto u(x)-t$ (which is defined only near the singular set of $M(0)$)
 and the signed distance function to $M(t)$ (which is smooth away from the singular set),
 we construct a function $w$ whose zero set at each time $t\in[0,\delta)$ is $M(t)$.
The function is smooth with nonvanishing gradient away from the singular set,
and, near the singular set, all of its level sets are weak set flows.
 (See the appendix~\ref{weak-set-flow-definition} for the definition of weak set flow.)
We then show (Theorem~\ref{criterion-theorem}) that the zero set of such any such function is a level set flow on some time interval $[0,\tau]$. In particular,
   $t\in [0,\tau] \mapsto M(t)$ is the level set flow of $M(0)$, $F_t(M(0))$. The same argument applies to $t\in [0,\tau] \mapsto M'(t)$, so $T_{\textnormal{disc}}\ge\tau>0$.     
\end{idea}  
\begin{proof}[Proof of Theorem~\ref{main_theorem}] 
We prove the second formulation of Theorem~\ref{main_theorem}.
By shifting time, we can assume that $T=0$.  Thus the flow starts at some negative initial time. Note that $t\mapsto M(t)$ is a weak set flow in $\RR^{n+1}$ 
(see Proposition~\ref{inner_most_prop}). 
If $x\in M(t)$ is of mean convex (resp.\ mean concave) type and if $\eps$ is as in Definition~\ref{concave type_def},
 then all of the points  in $M(\tau)\cap \BB(x,\eps)$, $\tau\in (t-\eps,0]$, 
 are also of mean convex (resp.\ mean concave) type.
Also, by the strong maximum principle, at
every regular or backwardly regular
point $x\in M(t)$ of mean convex type, the mean
 curvature is nonzero and points into $U(t)$.
Likewise, at  every regular or backwardly regular
  $x\in M(t)$ of mean concave type, the mean curvature is nonzero
and points out of $U(t)$.
If $x\in M(t)$ is of mean convex type and $y\in M(t)$ is of mean 
concave type, and if $\eps(x)$ and $\eps(y)$
are as in the definition, then $x\notin \BB(y,\eps(y))$ and $y\notin \BB(x,\eps(x))$, so
\begin{equation}\label{eq:disjoint-balls}
   \BB(x,\eps(x)/2)\cap \BB(y,\eps(y)/2) = \emptyset.
\end{equation}
We claim that all the backwardly regular points in $M(0)$ are in fact regular. 
To see this, suppose that $x\in M(0)$ is backwardly regular.
Then there is an $\eps>0$ such that $t\in (-\eps,0]\mapsto M(t)\cap \BB(x,\eps)$
is a smooth mean curvature flow.  Since $M(t)=M'(t)$ for $t<0$, we see that
$U(t)\cap \BB(x,\eps)$ lies on one side of $M(\tau)\cap \BB(x,\eps)$ and
$U'(t)\cap\BB(x,\eps)$ lies on the other side.   
By a local regularity theorem (Theorem~\ref{forward-regularity-theorem}), $(x,0)$
is regular.
(We will no longer need to refer to backward regularity or backward singularity.)
Let $\Sing_t$ denote the set of $x$ such that $(x,t)$ is a singular point.
Because each
  $x\in \Sing_0$ is of mean convex or mean concave type, there 
 exists an $\eps(x) >0$ as in the definition of mean convex/concave type. 
Because $\Sing_0$ is compact, it can be covered by finitely many balls $B(x_i,\eps(x_i)/5)$ with
$x_i\in \Sing_0$.
Let $W$ be an open set with smooth boundary such that
\[
\cup \BB(x_i, \eps(x_i)/4)
\subseteq  
W
\subseteq  
\cup \BB(x_i, \eps(x_i)/3).
\]
Note that all the points of $M(t)$ in $\cup_i \BB(x_i,\eps(x_i)/3) \setminus \cup_i \BB(x_i, \eps(x_i)/4)$
are regular points of $M(0)$.  Choose $W$ so that $\partial W$ is transverse to $M(0)$.
Let $\Pp$ be the union of those components of $W$ that contain points $x_i$ of mean convex type.
Let $\Nn$ be the union of the components of $W$ that contain points $x_i$ of mean concave type.
By~\eqref{eq:disjoint-balls}, $\Pp$ and $\Nn$ are disjoint.
Let $\eps=\min_i\eps(x_i)$.  Then for $-\eps\le t< t+\eta \le 0$,
\begin{equation}\label{eq:P-contain}
    U(t+\eta)\cap\overline{\Pp} \subseteq    U(t) \setminus M(t)
\end{equation}
and 
\begin{equation}\label{eq:N-contain}
   U(t)\cap \overline{\Nn} \subseteq   U(t+\eta) \setminus M(t+\eta).
\end{equation}
Since the spacetime singular set is 
closed, we can choose a $\delta$ with $0<\delta<\eps$
so that
\begin{equation}\label{eq:singular-set-contained}
   Q:=\cup_{|t|\le \delta}\, {\Sing}_t \subseteq \Pp\cup\Nn.
\end{equation}
We can also choose $\delta$ sufficiently small that for $t\in [-\delta,\delta]$, $M(t)$
is transverse to $\partial (\Pp\cup\Nn)$ and the mean curvature at every point
of $M(t)\cap \partial (\Pp\cup\Nn)$ is nonzero.  In particular, the mean curvature of $M(t)$
at every point of $M(t)\cap \partial \Pp$ is nonzero and points into $U(t)$, which implies that
\begin{equation}\label{eq:boundary-containment-P}
    U(t+\eta)\cap \partial \Pp \subseteq  U(t) \setminus M(t)
\     \text{for $t< t+\eta$ in $[-\delta,\delta]$}.
\end{equation}
Similarly,
\begin{equation}\label{eq:boundary-containment-N}
    U(t)\cap \partial \Nn \subseteq  U(t+\eta) \setminus M(t+\eta)
     \qquad 
     \text{for $t< t+\eta$ in $[-\delta,\delta]$}.
\end{equation}

\begin{claim}\label{containment-claim}
For $t< t+\eta$ in $[-\delta,\delta]$,
\begin{equation}\label{eq:P-nested}
 U(t+\eta) \cap \overline{\Pp} \subseteq  U(t) \setminus M(t)
\end{equation}
and
\begin{equation}\label{eq:N-nested}
 U(t)\cap \overline{\Nn} \subseteq U(t+\eta) \setminus M(t+\eta).
\end{equation}
In particular $M(t)\cap \overline{\Pp\cup\Nn}$ and $M(t+\eta)\cap \overline{\Pp\cup\Nn}$ are disjoint. 
\end{claim}
\begin{proof}
It suffices to prove it for $\eta<\delta$, since if the claim holds for $\eta$, it also
holds for every positive multiple of $\eta$.
Note that~\eqref{eq:P-nested} holds at time $t=-\delta$.  
Suppose that it fails at some later time.
Let $t$ be the first such 
time.  (There is a first time because
   $t\mapsto M(t)$ and $t\mapsto U(t)$ sweep out closed subsets of spacetime.)
Then by~\eqref{eq:boundary-containment-P}, $U(t+\eta)\cap \overline{\Pp}$ and $M(t)$
intersect in a nonempty, compact subset of $\Pp$.
But that contradicts the strong maximum principle 
(Theorem~\ref{maximum-principle-theorem})
applied to $t\mapsto U(t)\cap\Pp$ and $t\mapsto M(t+\eta)\cap\Pp$, which are weak set flows
in the space $\Pp$. 
This completes the proof of~\eqref{eq:P-nested}.  
The proof of~\eqref{eq:N-nested} is almost exactly the same.
The last assertion (``in particular\dots") follows since $M(\tau)\subseteq U(\tau)$.
\end{proof}
\begin{claim}\label{open-claim}
If $I$ is an open interval in $(-\delta,\delta)$, then the set
 \[
 \cup_{t\in I}M(t)\cap\, \overline{\Pp\cup\Nn}
 \]
is relatively open in $\overline{\Pp\cup\Nn}$.
\end{claim}
\begin{proof} 
For notational simplicity, we will prove it for $I=(-\delta, \delta)$; 
the general case is proved is proved in exactly
the same way.
We will show that the set
\begin{equation}\label{eq:P-part}
    A:=\cup_{t\in (-\delta,\delta)}M(t) \cap \overline{\Pp}
\end{equation}
is relatively open in $\overline{\Pp}$; the same argument shows that 
$\cup_{t\in (-\delta,\delta)}M(t)\cap\, \overline{\Nn}$ is relatively open in $\overline{\Nn}$.
Since $M(t)$ is closed and contains $\partial U(t)$, the set $U(t)\setminus M(t)$ is open.
Thus the set
\[
     \tilde A: = \left( (U(-\delta) \setminus M(-\delta)) \setminus U(\delta) \right)\cap \overline{\Pp}
\]
is relatively open in $\overline{\Pp}$.
If $t\in(-\delta,\delta)$, 
then by Claim~\ref{containment-claim}, $M(t)$ is disjoint from $U(\delta)\cap\overline{\Pp}$
and $M(t)\cap \overline{\Pp}$ is contained in $U(-\delta)\setminus M(-\delta)$.  Thus
\[
   A\subseteq \tilde A.
\]
The reverse inclusion is a consequence of the following sweeping-out property of the 
flow $t\mapsto M(t)$:
\begin{quote}
If $x\in U(t)\setminus U(t')$, then $x\in M(\tau)$ for some $\tau$ in the closed interval
between $t$ and $t'$.
\end{quote}
To see this property, note that since $(x,t)$ is in $\UU$ and $(x,t')$ is not,
the spacetime line segment joining those two points must 
contain a point in $\partial \UU$.
This completes the proof of Claim~\ref{open-claim}, since we have shown
 that $A=\tilde A$ and that $\tilde A$ is relatively open in 
the set  $\overline{\Pp}$.
\end{proof}
Set $X=\cup_{t\in (-\delta,\delta)}M(t)\cap (\Pp\cup\Nn)$ and 
let $u:X\rightarrow (-\delta,\delta)$ be the time-of-arrival function for the flow 
 $t\in (-\delta,\delta) \mapsto M(t)\cap (\Pp\cup\Nn)$:
\[
   \text{$u(x) = t$ for $x\in M(t)\cap (\Pp\cup\Nn)$ and $t\in (-\delta,\delta)$}.
\]
By Claim~\ref{open-claim}, the function $u$ is continuous.
Define $f: X\times \RR \to \RR$ as follows:
\[
f(x,t)
=
\begin{cases}
u(x)-t &\text{if $x\in \Pp$}, \\
t-u(x) &\text{if $x\in \Nn$}.
\end{cases}
\]
The set $f=0$ is the spacetime surface 
traced out by $t\in (-\delta,\delta) \mapsto M(t)\cap (\Pp\cup\Nn)$.
The  set $f=c$ is the spacetime surface traced out by 
\[
t\in (-\delta-c,\delta-c)\mapsto M(t+c)\cap\Pp
\]
 and by
\[
t\in (-\delta+c,\delta+c)\mapsto M(t-c)\cap\Nn. 
\]
Hence for each $c$,
\[
 t\mapsto \{f(\cdot,t)=c\}
\]
is a weak set flow in $\Pp\cup\Nn$.
Let $d(\cdot,t)$ be the signed distance function to $M(t)$:
\begin{equation*}
d(x,t)
=
\begin{cases}
\dist(x, M(t)) &\text{if $x\in U(t)$}, \\
-\dist(x,M(t)) &\text{if $x\notin U(t)$}.
\end{cases}
\end{equation*}
Let $G$ be an open set that contains $Q$ and whose closure is a compact
subset of $\Pp\cup\Nn$, where $Q$ is as in~\eqref{eq:singular-set-contained}.
Let $\phi:\RR^{n+1}\to [0,1]$ be a smooth function compactly supported
in $\Pp\cup\Nn$ such that $\phi\equiv 1$ on $\overline{G}$. 
Define $w:\big(X\cup \phi^{-1}(0)\big)\times  [0,\delta) \rightarrow \RR$ by 
\[
   w(x,t) = (1-\phi(x)) d(x,t) + \phi(x) f(x,t).
\]
Note that on $\phi^{-1}\big((0,1)\big)$ and for every $0 \leq t < \delta$,  the zero sets
 (resp.\ negative sets/positive sets) of $w$, $d$ and $f$ coincide.
Let $Z$ be an $\eps$-neighborhood of $M(0)\setminus G$, where $\eps$ is small enough
that $\overline{Z}$ is disjoint from $Q$, and that $w(\cdot,0)$ is smooth with nonzero 
gradient
 on $\overline{Z}$.  
 Choose $\tau\in [0,\delta)$  sufficiently small that
$w$ is smooth with non-vanishing gradient on $\overline{Z}\times[0,\tau]$, and that
\[
\cup_{t\in [0,\tau]}M(t)\cap G^c \subseteq  Z.
\]
Let 
\[
  Y = \cup_{t\in (-\delta,\delta)}M(t)\cap G.
\]
By Theorem~\ref{criterion-theorem} below, 
we can conclude that $t\in [0,\tau]\mapsto M(t)$
is the level set flow starting from $M(0)$. As the exact same argument would have applied for $M'(t)$, we get that $M'(t)=M(t)=F_t(M(0))$ for $t\in [0,\tau]$ and, in particular $T_{\textnormal{disc}}>0$.
\end{proof}
\begin{theorem}\label{criterion-theorem}
Suppose that $Y$ and $Z$ are bounded open subsets of $\RR^{n+1}$.
Suppose that $t\in [0,T]\mapsto M(t)$ is a weak set flow of compact sets in $Y\cup Z$.
Suppose that there is a continuous function 
\[
   w: \overline{Y\cup Z} \to \RR
\]
with the following properties:
\begin{enumerate}
\item $w(x,t)=0$ if and only if $x\in M(t)$.
\item For each $c$, 
\[
   t\in [0,T] \mapsto \{x\in Y: w(x,t)=c\} 
\]
defines a weak set flow in $Y$.
\item $w$ is smooth with non-vanishing gradient on $\overline{Z}$.
\end{enumerate}
Then $t\in[0,T]\mapsto M(t)$ is the level set flow of $M(0)$ in $\RR^{n+1}$. 
\end{theorem}
\begin{proof}
Let
\[
 \Phi[w] = w_t - |\nabla w|\,\Div \left( \frac{\nabla w}{|\nabla w|} \right).
\]
Now $\Phi[w]$ is a smooth function on $\overline{Z}\times[0,T]$ since $w$ is smooth
with nonzero gradient on that set.
Also, $\Phi[w]=0$ where $\{w=0\}$ on that set. (See Lemma~\ref{standard-lemma}).
Consider the function $\psi$ on $Z\times[0,T]$ given by
\[
  \psi = \frac{|\Phi[w]|}{|w|} \qquad\text{where $w\ne 0$}
\]
and
\[
  \psi =  \frac{|\nabla \Phi[w]|}{|\nabla w|} \qquad\text{where $w=0$.}
\]
The function is continuous, 
so on any compact subset $K$ of $Z\times[0,T]$, it attains a maximum value $c_K$,
and thus $|\Phi[w]|\le c_K|w|$ on $K$.  Hence by replacing $Z$ by a slightly smaller set, we can assume
that
\begin{equation}\label{eq:Phi-w-bound}
\text{$|\Phi[w]| \le C |w|$ on $\overline{Z}\times [0,T]$}
\end{equation} 
for some finite $C$.
Note that on $\overline{Z}\times [0,T]$,
\[
\Phi[e^{-\alpha t} w]
=
e^{-\alpha t} ( \Phi[w] - \alpha w ).
\]
Choose $\alpha>C$, where $C$ is as in~\eqref{eq:Phi-w-bound}.
Then where $w>0$,
\begin{equation}\label{eq:w-pos}
\begin{aligned}
\Phi[e^{-\alpha t} w]
&=
e^{-\alpha t} ( \Phi[w] - \alpha w ) \\
&\le
e^{-\alpha t} ( C w - \alpha w)
\\
&< 0.
\end{aligned}
\end{equation}
Likewise, 
$\Phi[e^{-\alpha t} w]
> 0$ where $w<0$.
Let $\tilde w(x,t)= e^{-\alpha t}w(x,t)$.

\begin{claim*}       
Let $c>0$.  The flow
\[
   t\in [0,T] \mapsto \{x: \tilde w(x,t)\ge c\} \tag{*}
\]
is a weak set flow in $Y\cup Z$. 
\end{claim*}

\begin{proof}[Proof of claim]
Suppose not. Then there is an interval $[a,b]\subseteq [0,T]$ and
a mean curvature flow $t\in [a,b]\mapsto S(t)$ of smooth closed surfaces in $Y\cup Z$ such
that $S(t)$ and~\thetag{*} are disjoint at time $t=a$ but not at
some later time $t\in [a,b]$. By replacing $[a,b]$ by a smaller interval, we can assume
that $S(t)$ and~\thetag{*} are disjoint for $t\in [a,b)$ but not for $t=b$.
Since $\Phi[\tilde w]<0$ where $\tilde w=c$ in $Z\times [0,T]$,
the flow
\[
  t\in [0,T] \mapsto \{x\in Z: \tilde w(x,t)\ge c\}
\]
is a weak set flow in $Z$.  (See Lemma~\ref{standard-lemma}\,\eqref{first-flow-item}.)
It follows immediately (see Remark~\ref{avoidance-remark}) that
\[
 S(b)\cap \{ \tilde{w}(x,t)\ge c\} \subseteq Y.
\]
But that is impossible according to Lemma~\ref{key-lemma} below.
This completes the proof of the claim.
\end{proof}
In the same way, if $c<0$, then
\begin{equation}\label{eq:other-side}
  t\in [0,T] \mapsto \{x: \tilde w(x,t)\le c\}
\end{equation}
is a weak set flow in $Y\cup Z$.
Since the union of weak set flows is a weak set flow, we see that
for $c>0$, 
\begin{equation}\label{eq:both}
\text{$t\in [0,T]\mapsto \{x: |\tilde w(x,t)|\ge c\}$ is a weak set flow in $Y\cup Z$.}
\end{equation}
(This flow is the union of the flows~\thetag{*} and~\eqref{eq:other-side}.)
Let $\eta$ be the minimum value of $|\tilde w|$  on $\partial (Y\cup Z)\times[0,T]$.  Then for
$0<c<\eta$,~\eqref{eq:both} implies that the flow
\begin{equation}\label{eq:big-barrier}
t\in [0,T]\mapsto \{x: |\tilde w(x,t)|\ge c\}\cup (Y\cup Z)^c = \{x: |\tilde w(x,t)| < c\}^c
\end{equation}
 is a weak set flow in $\RR^{n+1}$.
Let $M=M(0)$.  Since $F_t(M)$ and~\eqref{eq:big-barrier} are disjoint at time $0$, they remain 
disjoint for all $t\in [0,T]$.  Since this is true for all $c>0$,
\[
   F_t(M)\subseteq \{w(\cdot,t)=0\} = M(t)
\]
for all $t\in [0,T]$.
\end{proof}
\begin{lemma}\label{key-lemma}
Suppose that $Y$ is an open subset of a smooth Riemannian manifold.
Suppose that $w: Y\times [0,T]\to \RR$ is a continuous function such 
that for every $c$, 
\[
  t\in [0,T] \mapsto \{w(\cdot,t)=c\}
\]
is  weak set flow in $Y$.
Then for every $\alpha>0$ and for every $c>0$, the flow
\begin{equation*}\label{eq:first-set}
  t\in [0,T] \mapsto \{ e^{-\alpha t}w(\cdot,t) \ge c\}
\end{equation*}
is a weak set flow in $Y$.
Likewise, the flow
\begin{equation*}\label{eq:second-set}
   t\in [0,T] \mapsto \{ e^{-\alpha t} w(\cdot,t) \le -c\}
\end{equation*}
is a weak set flow in  $Y$.
\end{lemma}
\begin{proof}
It suffices to prove the first assertion, since the second assertion is the first assertion applied
to the function $-w$.
Let $Q(x,t)=e^{-\alpha t}w(x,t)$.  
Let $[a,b]\subseteq [0,T]$ and let $t\in[a,b]\mapsto S(t)$ be a smooth flow of closed surfaces.
Suppose that $Q(\cdot,a)<c$ on $S(a)$.  
We must show that $Q(\cdot,t)<c$ on $S(t)$ for all $t\in [a,b]$.
Suppose not. We may (by replacing $[a,b]$ by a shorter interval) assume that $b$ is the first time
when the inequality fails.  Thus
\begin{equation}\label{eq:Q-equation}
\begin{aligned}
&\text{$Q(x,t) = e^{-\alpha t} w(x,t) < c$ for all $x\in S(t)$ and $t\in [a,b)$}, \\
&\text{$Q(y,b) =  e^{-\alpha b} w(y,b) = c$ for some $y\in S(b)$.}
\end{aligned}
\end{equation}
Let $\hat{c}= e^{\alpha b}c$.  By~\eqref{eq:Q-equation},
\begin{equation}\label{eq:w-equation}
\begin{aligned}
  &w(x,t) < e^{\alpha t}c < e^{\alpha b}c =\hat{c} \qquad\text{for $t\in [a,b)$, $x\in S(t)$}, \\
  &w(y,b) = e^{\alpha b}c = \hat{c} \qquad\text{for some $y\in M(b)$.}
\end{aligned}
\end{equation}
But this contradicts the fact that $\{w=\hat{c}\}$ defines a weak set flow:
by~\eqref{eq:w-equation}, the flows $t\mapsto S(t)$ and $t\mapsto \{w(\cdot,t)=\hat{c}\}$ are disjoint
for $t\in [a,b)$ but not for $t=b$.
\end{proof}
\begin{lemma}\label{standard-lemma}
Let $Z$ be an open subset of Euclidean space
 and let $u: Z\times [0,T]\to \RR$ be a smooth function with non-vanishing gradient.
Let 
\[
 \Phi[u] = u_t - |\nabla u|\,\Div \left( \frac{\nabla u}{|\nabla u|} \right).
\]
Let $c\in \RR$.
\begin{enumerate}
\item\label{first-flow-item} The flow 
\[
 t\mapsto \{x\in Z: u(x,t)\ge c\}
\]
 is a weak set flow in $Z$ if and only $\Phi[u]\le 0$ at all spacetime points where $u=c$.
\item\label{second-flow-item} The flow 
\[
 t\mapsto \{x\in Z: u(x,t)\le c\}
\]
 is a weak set flow in $Z$ if and only $\Phi[u]\ge 0$ at all spacetime points where $u=c$.
\item\label{third-flow-item} The flow 
\[
 t\mapsto \{x: u(x,t)=c\}
\]
is a mean curvature flow if and only if $\Phi[u]=0$ at all spacetime points where $u=c$.
\end{enumerate}
\end{lemma}
\begin{proof}
The normal velocity at $(x,t)$ of the moving surfaces $\tau\mapsto \{ u(\cdot,\tau)=c\}$
is
\begin{equation}\label{eq:velocity}
  \vv(x,t)= \frac{-\nabla u}{|\nabla u|^2} u_t.
\end{equation}
The mean curvature $H(x,t)$ of $\{u(\cdot,t)=c\}$ at $x$ is
\begin{equation}\label{eq:H}
 H = -\Div\left( \frac{\nabla u}{|\nabla u|} \right) \, \frac{\nabla u}{|\nabla u|}.
\end{equation}
\newcommand{\nn}{\mathbf{n}}
Let $\nn=\frac{\nabla u}{|\nabla u|}$ be the unit normal vector to $\{u(\cdot,t)=c\}$
that points into $\{u(\cdot,t)>c\}$.  Then from~\eqref{eq:velocity} and~\eqref{eq:H}, we see that
\begin{equation}\label{eq:fast}
  \vv\cdot \nn \ge H\cdot \nn
\end{equation}
if and only if
\[
   -\frac{u_t}{|\nabla u|} \ge -\Div\left( \frac{\nabla u}{|\nabla u|} \right),
\]
i.e., if and only if $\Phi[u]\le 0$.   This proves (1).
Assertions (2) and (3) are proved in the same way. (Alternatively, (2) follows by applying (1)
to the function $-u$, and (3) follows immediately from (1) and (2).)
\end{proof}
\begin{remark}\label{avoidance-remark}
Note that in case (1) of the lemma, the set 
\[
\{x\in Z: u(x,t)\ge c\} \tag{*}
\]
is a smooth manifold-with-boundary in 
the space $Z$, and the boundary $\{x\in Z: u(x,t)=c\}$ 
is moving smoothly.  Thus, assuming $\Phi[u]\le 0$, if $t\in [0,T]\mapsto S(t)$ is a smooth mean curvature flow of 
hypersurfaces properly embedded in $Z$, and if $S(t)$ and $\thetag{*}$ are disjoint for $t\in [0,T)$,
then they are also disjoint at time $T$
 by the strong maximum principle for smooth flows. (See~\eqref{eq:fast}.)
If the strict inequality $\Phi[u]<0$ holds, which is the case when we apply Lemma~\ref{standard-lemma}
 in the proof of Theorem~\ref{criterion-theorem}, 
 then we have strict inequality in~\eqref{eq:fast}, which
 immediately implies disjointness of $S(t)$ and \thetag{*} at time $T$.
 (This is the most elementary case of the maximum principle.)
\end{remark}
\section{Mean convex blow up type implies mean convex type}\label{blowupimplies}
In this section we prove Theorem \ref{alt}, which shows that a singular point of mean convex (or mean concave) blow-up type is of mean convex (or mean concave) type.
\begin{proof}[Proof of Theorem~\ref{alt}]
Suppose that $(x_0,t_0)$ is of mean convex blowup type. (The proof in the mean concave case is 
essentially the same.)
If $(x,t)$ is a regular point of the flow, let $H(x,t)$ be the mean curvature in the direction
of the unit normal that points into $U(t)$ and let $R(x,t)$ be
the supremum of $r>0$ such that the flow is smooth in
\[
      \BB(x,r) \times [t-r^2,t]
\]
and such that the norm of the second fundamental form is bounded by $1/r$ in that set.
We assert that there exist an $\eps>0$ and an $a>0$ such that
at each point $(x,t)$ of the flow in the set
$
  W:= \BB(x_0,\eps)\times [t_0-\eps^2,t_0]
$,
\begin{equation}\label{gauss} 
\text{The Gaussian density at $(x,t)$ is less than $2$},
\end{equation}
and
\begin{equation}\label{HR}  
\text{If $(x,t)$ is a regular point, then $H(x,t)R(x,t) \ge a$.}
\end{equation}
To see that we can choose such $\eps$ and $a$, note that~\eqref{gauss} holds (for small $\eps$)
because the Gaussian density at $(x_0,t_0)$ is less than $2$ (since the tangent flow is a shrinking
sphere or cylinder) and because Gaussian density is upper semicontinuous.
To see~\eqref{HR}, let $\alpha(k)$ be the infimum of $H(x,t)R(x,t)$
among regular points $(x,t)$ in $\BB(x_0,1/k)\times [-1/k^2,0]$,
let $\alpha=\lim_k\alpha(k)$, and 
let $(x_k,t_k)$ be  regular
points of the flow converging to $(x_0,t_0)$ with $t_k\le t_0$ such that
\[
   H(x_k,t_k)R(x_k,t_k)\to \alpha.
\]
By passing to a subsequence (see Definition~\ref{blow-up-type}\,\eqref{blow-up-type-2}),
 we can assume that the flows
\[
   t \in [-Q_k^2\, t_k,0] \mapsto Q_k(U(t_k + Q_k^{-2}t)-x_0) \\
\]
(with $Q_k=|A(x_k,t_k)|$) 
 converge smoothly to a flow $t\in(-\infty,0]\mapsto N_t$ where the $N_t$
 are convex regions with smooth boundaries. 
Let $H$ be the mean cuvature of $N_0$ at the origin (which is nonzero since the norm
of the second fundamental form there is $1$ and since $N_0$ is convex).
Let $R$ be the supremum of radii $r>0$ such that the norm of the second fundamental
form in $N_t\cap \BB(0,r)$ is bounded by $1/r$ for $t\in [-r^2,0]$.  By smoothness, $R>0$.
By the smooth convergence,
\[
 \alpha= \liminf_{k\to\infty} H(x_k,t_k)R(x_k,t_k) = HR >0.
\]
Now let $0<a<\alpha$ (e.g., $a=\alpha/2$).  It follows immediately
that~\eqref{HR} holds if $\eps>0$ is sufficiently small.

\setcounter{claim}{0}
\begin{claim}\label{sphere/cylinder}
At every singular point in $W$, each tangent flow is a multiplicity-one shrinking
sphere or cylinder.
\end{claim}
To prove Claim~\ref{sphere/cylinder},
let $t\mapsto (|t|^{1/2})_\#V$ be a tangent flow at a singular point in $W$.
By~\eqref{gauss}, the multiplicity-one regular points are dense in the tangent flow.
 By the local regularity theorem~\cite{white_regularity}, the convergence
of the dilated flows to the tangent flow is smooth in a spacetime neighborhood of each
regular point
of the tangent flow.   Thus by~\eqref{HR},
 if $y$ is a regular point of $V$ with mean curvature $H$, then
$V$ is smooth in $\BB(y, a H^{-1})$.  (Indeed, the norm of the second fundamental form is bounded
by $a^{-1}H$ in that ball.)  Since $H\le |y|/2$, $V$ is smooth in $\BB(y, 2a\,|y|^{-1})$.  Since such regular
points are dense in $\spt(V)$, $V$ is smooth everywhere.
By choice of $\eps$ (see~\eqref{HR}), the mean curvature of $V$ is everywhere $\ge 0$.  
Hence $V$ is a shrinking sphere or cylinder \cite[Thm. 10.1]{CM_generic}.
\begin{claim}\label{weak-containment}
If $t_0-\eps < t_1 < t_2 \le t_0$, then
\[
    U(t_2)\cap \BB(x_0,\eps) \subseteq U(t_1).
\]
\end{claim}
To prove Claim~\ref{weak-containment}, 
suppose to the contrary that there is point $y\in U(t_2)\cap\BB(x_0,\eps)$ that is not in $U(t_1)$.
Let
\[
  0< \delta < \dist(y, U(t_1) \cup \partial \BB(x_0,\eps)).
\]
Let $t\in (t_1,t_0)$ be the first time $\ge t_1$ such that 
\[
  \dist(y, U(t)) = \delta.
\]
Let $p\in U(t)$ be a point such that $\dist(y,p)=\dist(y,U(t))$.
Note that the tangent flow at $(p,t)$ lies in a halfspace (namely the halfspace $\{x: x\cdot (y-p) \le 0\}$).
Hence $(p,t)$ is a regular point of the flow by Claim~\ref{sphere/cylinder}.
Now the mean curvature at $(p,t)$ is nonzero and points into $U(t)$, i.e., in the direction of $p-y$.
It follows that for $\tau<t$ very close to $t$, $\dist(y, U(\tau)) < \delta$, a contradiction.
This completes the proof of Claim~\ref{weak-containment}.

It remains only to show that for $t_0-\eps^2 \le t_1 < t_2 \le t_0$, if 
\[
   x\in \BB(x_0,\eps)\cap U(t_2),
\]
then $x$ is in the interior of $U(t_1)$.
If $x$ is the interior of $U(t_2)$, then (by Claim~\ref{weak-containment}) it is in the interior of $U(t_1)$.
Thus we may assume that $x$ is in the boundary of $U(t_2)$.
For $\tau \in [t_1,t_2)$ sufficiently close to $t_2$, $x$ is in the interior of $U(\tau)$.
If $(x,t_2)$ is a regular point, this is because the mean curvature is nonzero and points
into $U(t_2)$.  If $(x,t_2)$ is a singular point, this is true by Claim~\ref{sphere/cylinder}.
Since $x$ is in the interior of $U(\tau)$ and since $U(\tau)\subseteq U(t_1)$, it
follows that $x$ is in the interior of $U(t_1)$.  This completes the proof of Theorem~\ref{alt}.
\end{proof}
\appendix
\renewcommand{\thetheorem}{\thesection\arabic{theorem}}
\setcounter{theorem}{0}
\section{Weak set flows}
In this appendix, we collect some results on weak set flows.
\begin{definition}\label{weak-set-flow-definition}
Let $W$ be an open subset of a Riemannian manifold and $I\subseteq \RR$ be an interval.
A family
\[
  t\in I\mapsto M(t)
\]
of subsets of $W$ is called
a {\bf weak set flow} in $W$ provided:
\begin{enumerate}
\item $\{(x,t): t\in I, \, x\in M(t)\}$ is a relatively closed subset of $W\times I$.
\item If $[a,b]\subseteq I$, if $t\in [a,b]\mapsto S(t)\subseteq W$ is a classical mean
curvature flow of smooth, closed hypersurfaces, and if $S(a)$ is disjoint from $M(a)$,
then $S(t)$ is disjoint from $M(t)$ for all $t\in [a,b]$.
\end{enumerate}
\end{definition}
In~\cite{white_topology}, the definition of weak set flow is slightly more complicated because
it generalizes the notion of mean curvature flow of smooth surfaces with boundary,
whereas in this paper we are concerned with flow of surfaces without boundary.
\begin{theorem}\label{maximum-principle-theorem}
Suppose that
\begin{equation*}
   t\in [0,T] \mapsto M_i(t) \tag{*}
\end{equation*}
is a weak set flow in $U$ for $i=1,2$, where $U$ is the interior of a compact
subset of smooth Riemannian manifold $N$.
Let 
\[
  \MM_i = \{(x,t): t\in [0,T], \, x\in M_i\}
\]
be the spacetime set swept out by the flow~\thetag{*}.
Suppose also that 
\[
  \overline{\MM_1}\cap \overline{\MM_2}
\]
is a compact subset of $U\times (0,T]$.
Then $\MM_1$ and $\MM_2$ are disjoint.
\end{theorem}
\begin{proof}
Let $\Gamma_i = \overline{\MM_i}\setminus \MM_i$.
One can think of the flow~\thetag{*} as a flow of (generalized) surfaces-with-boundary in $N$,
where the boundary at time $t$ is $\{x: (x,t)\in \Gamma_i\}$.
In the terminology of~\cite{white_topology}, 
\[
 t \in [0,T] \mapsto \{x: (x,t)\in \overline{\MM_i}\}
\]
is a weak set flow in $N$ generated
by the spacetime set $\Gamma_i\cup (M_i(0)\times \{0\})$.  Theorem~\ref{maximum-principle-theorem}
 is a special case of Theorem 7.1 of that paper.
\end{proof}
Given a relatively closed set $M$ of $W$, there is a (unique) weak set flow
\[
  t\in [0,\infty)\mapsto F_t(M)  
\]
in $W$ for which $F_0(M)=M$ and for which the following
property holds: if $t\in [0,T]\mapsto S(t)$ is any weak set flow in $W$ with $S(0)\subseteq M$,
then $S(t)\subseteq F_t(M)$ for all $t\in [0,T]$.   The flow $t\in[0,\infty)\mapsto F_t(M)$
is the {\bf level set flow} starting at $M$.
\begin{proposition}\label{inner_most_prop}
Suppose that $U$ is any closed region in a Riemannian manifold $N$.
Let
\[
  \UU := \{(x,t): t\ge 0, \,\, x\in F_t(U) \}
\]
be the spacetime region swept out by $t\mapsto F_t(U)$, and let
\[
  M(t) = \{x: (x,t)\in \partial \UU\}.
\]
Then $t\mapsto M(t)$ is a weak set flow.
\end{proposition}
\begin{proof}
We must show that if $t\in [a,b]\mapsto S(t)$ is a smooth flow
of connected, closed surfaces with $S(a)$ disjoint from $M(a)$,
then $S(t)$ is disjoint from $M(t)$ for all $t\in [a,b]$.
Trivially, $M(t)$ contains $\partial U(t)$.  Thus either $S(a)\subseteq U(a)\setminus M(a)$
or $S(a)$ is disjoint from $U(a)$.  In the latter case, $S(t)$ is disjoint from $U(t)$
for all $t\in [a,b]$ (since $t\mapsto U(t)$ is a weak set flow)
 and therefore disjoint from $M(t)$ since $M(t)\subseteq U(t)$.
Thus it suffices to prove it when $S(a)\subseteq U(a)\setminus M(a)$.
If $G$ is a relatively open subset of spacetime $N\times[0,\infty)$, let
$G^*$ be the union of all spacetime sweepouts of smooth flows
\[
  t\in [c,d] \mapsto \Sigma(t)
\]
of smooth closed surfaces such that $\Sigma(c)\times \{c\} \subseteq G$.
Note that $G^*$ is relatively open in $N\times[0,\infty)$.
By definition of $F_t$, if $G\subseteq \UU$,
then $G^*\subseteq \UU$.
In particular, letting $G$ be the interior (relative to $N\times[0,\infty)$)
of $\UU$, we see
that $G^*$ is a relatively open subset of $N\times[0,\infty)$
contained in $\UU$, and thus that $G^*$
is disjoint from $\partial \UU$.  By definition of $G^*$, the spacetime
sweepout of $t\in [a.b]\mapsto S(t)$ is contained in $G^*$, so $S(t)$ is disjoint
from $M(t)$ for all $t\in [a,b]$.
\end{proof}
\section{The outermost Brakke flow}
In this section, we prove a theorem (Theorem~\ref{forward-regularity-theorem})
 that, in certain situations, allows one to deduce regularity from backward regularity.
We will need the following basic facts about level set flow:
\begin{lemma}\label{properties-lemma}\qquad
\begin{enumerate}
\item\label{lemma-nested-item}
 If $U_1\supseteq U_2 \supseteq\dots$ are compact sets, then $F_t(\cap_iU) = \cap_i F_t(U_i)$.
\item\label{set-avoidance-item}{\rm [Set avoidance]}
If $K$ is compact, $C$ is closed, and $K$ and $C$ are disjoint, then $F_t(K)$ and $F_t(C)$
are disjoint for all $t>0$.  
\item\label{varifold-avoidance-item} {\rm [Varifold avoidance]} 
If $t\in [0,\infty)\mapsto \mu(t)$ is a Brakke flow of $n$-varifolds in $\RR^{n+1}$, then the spacetime
support of the flow is contained in the set
\[
  \{(x,t): t\ge 0, \, x\in F_t(\spt(\mu(0)) \}
\]
Equivalently, if $\spt(\mu(0))$ is disjoint from a compact set $K$, then the spacetime support of the
flow $t\mapsto \mu(t))$ is disjoint from
\[
  \{(x,t): t\ge 0, \, x\in F_t(K)\}.
\]
\end{enumerate}
\end{lemma}
 \begin{proof} 
   Assertion~\eqref{lemma-nested-item}
   follows immediately from the definition.  
   Assertion~\eqref{set-avoidance-item} is a special
   case of Theorem~\ref{maximum-principle-theorem}.
    See~\cite{Ilmanen}*{10.7} for~\eqref{varifold-avoidance-item}.
   
 \end{proof}
 
\begin{theorem}\label{two-defs-theorem}
Suppose $M$ is a smoothly embedded, closed hypersurface.
Let $U$ be the compact region it bounds, and let $t\mapsto M(t)$
be the outer flow for $M$ (see~\eqref{eq:definition-inner-outer}).
Then for $T>0$,
\[
  M(T) = \lim_{\tau\uparrow T}\partial F_\tau(U).
\]
\end{theorem}
\begin{proof}
For every $\eps>0$, $\partial F_\tau(U)$ lies in the $\eps$-neighborhood
of $M(T)$ for all $\tau<T$ sufficiently close to $T$ since $\partial F_\tau(U)$
is contained in $M(\tau)$ and since $t\mapsto M(t)$ traces out a closed subset of spacetime.
Conversely, we claim that $M(T)$ lies in the $\eps$-neighborhood of $\partial F_\tau(U)$
for all $\tau<T$ sufficiently close to $T$.  For suppose not. Then
there is a sequence $\tau_i\uparrow T$, a point $x\in M(T)$, and an $r>0$ such that
\begin{equation}\label{eq:away}
   \dist(x, \partial F_{\tau_i}U) > r > 0
\end{equation}
for all $i$.  Let $\BB$ be the closed ball of radius $r$ centered at $x$.
Fix a $\tau=\tau_i$ sufficiently close to $T$ that
 $x$ is in the interior
of $F_{T-\tau}(\BB)$.  Thus $(x,T)$ is in the interior
of the spacetime track $\mathcal{B}$ of $t\in [\tau,\infty)\mapsto F_{t-\tau}(\BB)$.
Note that $\BB$ cannot be contained in $F_\tau(U)$, since then $\mathcal{B}$ would
be in $\UU$,  so $(x,T)$ would be in the interior of $\UU$, 
which is impossible since $(x,T)\in \partial \UU$.
Likewise, $\BB$ cannot be disjoint from $F_\tau(U)$, since otherwise
 $F_{T-\tau}(\BB)$ would be disjoint from $F_T(U)$, which is impossible
since $x\in F_T(\BB) \cap M(T)$ and since $M(T)\subseteq F_T(U)$.
Since $\BB$ is not contained in $F_\tau(U)$ or its complement, it must contain
a point in $\partial F_\tau(U)$, contradicting~\eqref{eq:away}.
\end{proof}
 
In the following theorem, we assume that $U\subseteq\RR^{n+1}$ is a compact region with
 $\Hh^n(\partial U)<\infty$.
We choose compact regions $U_i$ with smooth boundaries such that
\begin{enumerate}
\item For each $i$, $U_{i+1}$ is contained in the interior of $U_i$.
\item $\cap U_i=U$.
\item $\sup_i\Hh^n(\partial U_i)<\infty$.
\end{enumerate}
By perturbing each $U_i$ slightly, we can also assume that
\begin{enumerate}
\setcounter{enumi}{4}
\item $\partial U_i$ never fattens.
\end{enumerate}
By passing to a subsequence, we can assume that the measures $\Hh^n\llcorner \partial U_i$ converge
weakly to a radon measure $\mu$.  Of course $\mu$ is supported in $\partial U$.
We can also assume that
\begin{enumerate}
\setcounter{enumi}{5}
\item\label{one-two-item}
If $W$ is an open set and $W\cap \partial U$ is a smooth, connected $n$-manifold, 
then $\mu$ coincides in $W$ with $\Hh^n\llcorner \partial U$ or with $2\,\Hh^n\llcorner \partial U$
 according to whether $W\cap \partial \Omega$ is or is not contained
in the closure of the interior of $U$.
\end{enumerate}
We achieve~\eqref{one-two-item} by choosing the $U_i$ so that $W\cap U_i$ is smooth and converges smoothly
to $W\cap\partial U$.  Note that the convergence is with multiplicity $1$ or $2$ according to whether $W\cap \partial U$
is or is not in the closure of the interior of $U$.
\begin{theorem}\label{outer-flow-theorem}
There is an integral Brakke flow $t\in[0,\infty)\mapsto \mu(t)$ 
such that $\mu(0)=\mu$ and such that the spacetime support of the flow
is the spacetime set swept out by $t\in[0,\infty)\mapsto M(t)$,
where $t\mapsto M(t)$ is the outer flow for $M$ (see~\eqref{eq:definition-inner-outer}).
That is, for $t>0$, the Gauss density of the flow $\mu(\cdot)$ at $(x,t)$ is $>0$ if and only if 
$x\in M(t)$.
\end{theorem}
\begin{proof}
Let $\UU$  and $\MM_i$ be the spacetime tracks of the flows
$t\mapsto F_t(U)$ and $t\mapsto F_t(\partial U_i)$.
By definition, $\partial \UU$ is the spacetime track of $t\mapsto M(t)$.
Using elliptic regularization, we can find integral
Brakke flows $t\in [0,\infty)\mapsto \mu_i(t)$ starting from $\partial U_i$.
(That is, the initial Radon measure is $\Hh^n\llcorner \partial U_i$.) 
By passing to a subsequence, we can assume that they converge to an integral 
Brakke flow $t\mapsto \mu(t)$ with $\mu(0)=\mu$.
Let $\VV_i$ and $\VV$ be the spacetime supports of the flows $t\mapsto \mu_i$ and
$t\mapsto \mu(t)$.
By Lemma~\ref{properties-lemma}\,\eqref{varifold-avoidance-item}, 
\begin{equation}\label{eq:V-in-U}
  \VV\subseteq \UU.
\end{equation}
By the same lemma, $\VV_i\subseteq \MM_i$, which is disjoint from $\UU$ 
(by Lemma~\ref{properties-lemma}\,\eqref{set-avoidance-item})
 and therefore disjoint from $\interior(\UU)$.   Passing to the limit,
$\VV$ is disjoint from the interior of $\UU$, so by~\eqref{eq:V-in-U},
\[
  \VV\subseteq \partial \UU.
\]
Because it comes from elliptic regularization (and because $\partial U_i$ does not fatten),
the Brakke flow $\mu_i(\cdot)$ has the following property:
for every $t>0$ \cite[11.2,11.4]{Ilmanen}, 
\begin{equation}\label{eq:reduced-containment}
 \partial^*(F_t(U_i)) \subseteq \spt(\mu_i(t)) \subseteq V_i(t).
\end{equation}
For any closed set of finite perimeter, the closure of the reduced boundary is
equal to the boundary of the interior \cite[Theorem~4.4]{Giusti_MS_book}. 
Thus~\eqref{eq:reduced-containment} implies
\[
  \partial ({\rm interior}(F_t(U_i)) \subseteq V_i(t).
\]
Now suppose that $x\in \partial F_t(U)$.  Then $x$ is the interior of $F_t(U_i)$ for all $i$.
Let $\eps>0$.
By assertion~\eqref{lemma-nested-item} of Lemma~\ref{properties-lemma},  for all sufficiently large $i$, $\BB(x,\eps)$ contains
a point not in $F_t(U_i)$.  Thus $\BB(x,\eps)$ contains a point in $\partial({\rm interior}(F_t(U_i))$
and therefore in $V_i(t)$. 
Letting $i\to \infty$, we see that $\overline{\BB(x,\eps)}$ contains a point in $V(t)$.
Since $\eps>0$ is arbitrary, $x\in V(t)$.  We have shown that $\partial F_t(U)\subseteq V(t)$
for all $t$.
Since $M(t) = \lim_{\tau\uparrow t}\partial F_\tau(U)$, it follows that $M(t)\subseteq V(t)$.
\end{proof}
The Brakke flow constructed in Theorem~\ref{outer-flow-theorem} has an additional property
called unit regularity:
\begin{definition}\label{unit-regular-definition}
A {\bf unit-regular} Brakke flow is an integral Brakke flow such that every
spacetime point of Gaussian density one is regular (and not just backwardly regular).
\end{definition}
In arbitrary integral Brakke flows, spacetime points of Gauss density  $1$
may fail to be regular because of sudden vanishing.  
For example, in  a non-moving, multiplicity-one plane $P$ that
vanishes at time $T$, the points $(x,T)$ with $x\in P$ are all backwardly regular but not
regular.
\begin{theorem}\label{unit-regular-theorem}
The Brakke flow constructed in the proof of Theorem~\ref{outer-flow-theorem} is unit-regular.
\end{theorem}
\begin{proof}
Let $\Cc$ be the class of unit-regular Brakke flows.
  By Allard's theorem, this class includes translators
for mean curvature flow,
 since such translators are stationary integral varifolds for a certain Riemannian metric.
The local regularity theory of \cite{white_regularity} implies that the class $\Cc$ is closed under weak convergence
of Brakke flows: see \cite{white_schulze}*{Theorem~4.2}. 
Hence all flows obtained by elliptic regularization are unit-regular
(because they are limits of translating flows), as are all limits of such flows.
\end{proof}
\begin{theorem}\label{forward-regularity-theorem}
Suppose that $U\subseteq\RR^{n+1}$ is a compact region with $\Hh^n(\partial U)<\infty$, 
and let $t\mapsto M(t)$ be the outer flow for $M=\partial U$ 
(see~\eqref{eq:definition-inner-outer}).
Suppose that $T>0$ and that $p\in M(T)$ is backwardly regular for the flow $M(\cdot)$.
If $p$ is in the closure of the interior of $F_T(U)$, then $(p,T)$ is a regular point of the flow.
\end{theorem}
\begin{proof}
By hypothesis, there is an open set $W$ containing $p$ and a time $a<T$ such that
$t\in [a,T]\mapsto W\cap M(t)$
is a smooth mean curvature flow of smoothly embedded hypersurfaces.  
By replacing $W$ by a smaller open set 
 and by replacing $[a,T]$ by a smaller time interval  $[a',T]$,
we can assume that $W\cap M(t)$ is connected and nonempty for all $t\in [a,T]$.
It follows that
$W\cap M(t)$ is contained in the closure of the interior of $F_t(U)$ 
for all $t\in [a,T]$.   
Thus we can apply Theorems~\ref{outer-flow-theorem} and~\ref{unit-regular-theorem}
 with $a$ as the initial time to get a
unit-regular Brakke flow
\[
  t\in [a,\infty) \mapsto \mu(t)
\]
such that $\mu(a)$ coincides with $\Hh^n\llcorner M(a)$ in $W$
and such that 
\begin{equation}\label{eq:support}
\text{The spacetime support of the Brakke flow is $\{(x,t): t\ge a, \, x\in M(t)\}$.}
\end{equation}
For almost all $t$, the varifold corresponding to $\mu(t)$ has locally bounded first variation,
which implies that
\begin{equation}\label{eq:coincide}
   \mu(t)\llcorner W = k(t)\, \Hh^n \llcorner (W\cap M(t))
\end{equation}
for some nonnegative integer $k(t)$. By~\eqref{eq:support},
\begin{equation}\label{eq:at-least-one}
  \text{$k(t)\ge 1$ for almost all $t\in [a,T]$}.
\end{equation}
Also, for every $t\ge a$,
\begin{equation}\label{eq:left-right-limits}
    \lim_{\tau\uparrow t} \mu(\tau) \ge \mu(t) \ge \lim_{\tau\downarrow t}\mu(\tau).
\end{equation}
(The limits exist and satisfy the inequality.)
Since $k(a)=1$, we see from~\eqref{eq:coincide},~\eqref{eq:at-least-one}, 
and~\eqref{eq:left-right-limits} that \eqref{eq:coincide} holds with $k(t)=1$
for every $t\in [a,T]$.
Hence the Gaussian density at $(p,T)$ is one.  Since the Brakke flow $t\mapsto \mu(t)$
is unit-regular, $(p,T)$ is a regular point of the flow.
\end{proof}
\section{Additional Results about Inner and Outer Flows}
In this section, we prove that the $M(t)=\partial F_t(U)$  except for countably many $t$
((where $M(\cdot)$ is the outer
flow for $M=\partial U$),
and we prove that for a generic starting surface $M$, the inner and outer flows are the same
(i.e, $T_\textnormal{disc}=\infty$.)
Both proofs are based on the following general fact about metric spaces:
if $X$ is a separable metric space and if $f:X\to\RR$ is continuous, then 
\begin{equation}\label{eq:countable}
\text{$\{f(x): \text{$f$ has a local maximum at $x$}\}$ is countable}.
\end{equation}
\begin{theorem}\label{countable-times-theorem}
Let $U$ be a closed region in $\RR^{n+1}$, 
and let $t\mapsto M(t)$ be the outer flow
for $M=\partial U$ (See~\eqref{eq:definition-inner-outer}).  
Then $M(t)=\partial F_t(U)$ for all but countably
many $t$.
\end{theorem}
\begin{proof}
As usual, let $\UU$ be the spacetime track of $t\mapsto F_t(U)$, so that $\partial \UU$
is the spacetime track of $t\mapsto M(t)$.
Suppose that $p\in M(T)\setminus \partial F_T(U)$.  Then $p$ is the interior of $F_T(U)$,
so $F_T(U)$ contains a ball $\BB=\BB(p,r)$.  Note that $F_{t-T}(\BB) \subseteq F_t(U)$
for $t>T$.  Thus the time function $(x,t)\mapsto t$ restricted to $\partial \UU$ has a local
maximum at $(p,T)$, so the Theorem~\ref{countable-times-theorem}
is a special case of~\eqref{eq:countable}.
\end{proof}
\begin{theorem}
Let $\phi:\RR^{n+1}\to \RR$ be a proper continuous function.
For all but countably many $s\in\RR$, the inner and outer flows
for $M_s:=\{\phi=s\}$ coincide.
\end{theorem}
\begin{proof}
Define $\Phi:\RR^{n+1}\times[0,\infty)\to \RR$ by
\[
          \text{  $\Phi(x,t) = s$ if $x\in F_t(M_s)$.}
\]
Thus $\{\Phi=s\}$ is the spacetime track of $F_t(M_s)$.
Let $U_s=\{\phi\le s\}$ and $U'_s=\{\phi\ge s\}$.
Then $\{\Phi\le s\}$ is the spacetime track $\UU_s$ of $F_t(U_s)$ 
and $\{\Phi\ge s\}$ is the spacetime track $\UU'_s$ of $F_t(U'_s)$.
The assertion of the theorem is that $\partial \UU_s=\partial \UU'_s$
 for all but countable many $s$.
 
Suppose that $(p,t)\in \partial\UU_s\setminus \partial \UU_s'$.
That is, $(p,t)$ is in the boundary of $\{F\le s\}$ but not in the boundary of $\{F\ge s\}$.
Then $F:\RR^{n+1}\times[0,\infty)\to \RR$ has a local minimum at $(p,t)$.
Similarly, $F$ has a local maximum at each point of 
   $\partial \UU_s'\setminus \partial \UU_s$. 
By~\eqref{eq:countable}, $\{s: \partial \UU_s\ne \partial \UU'_s\}$ is countable.
\end{proof}
\bibliography{HershkovitsWhite}
\bibliographystyle{alpha}
  
\end{document}